\documentclass[letterpaper, 10 pt, conference]{ieeeconf}  

\IEEEoverridecommandlockouts                              

\usepackage{graphics} 
\usepackage{epsfig} 
\usepackage{mathptmx} 
\usepackage{times} 
\usepackage{amsmath} 
\usepackage{amssymb}  

\newcommand{\argmin}{\mathop{\rm arg~min}\limits}

\newcommand{\prox}{\mathrm{prox}}
\newcommand{\fin}{\triangleleft}
\newcommand{\R}{\mathbb{R}}
\newcommand{\Z}{\mathbb{Z}}
\usepackage{mathtools}

\usepackage{xcolor}

\usepackage{comment}
\usepackage{algorithmic}
\usepackage{algorithm}

\newtheorem{definition}{Definition}
\newtheorem{assumption}{Assumption}
\newtheorem{lemma}{Lemma}

\newtheorem{theorem}{Theorem}

\newtheorem{example}{Example}

\usepackage{bbm}

\title{\LARGE \bf
Continuation Method for Nonsmooth Model Predictive Control Using Proximal Technique
}

\author{Ryotaro Shima$^{1}$, Ryuta Moriyasu$^{1}$, and Teruki Kato$^{1}$
\thanks{$^{1}$Ryotaro Shima, Ryuta Moriyasu, and Teruki Kato are with Toyota Central R\&D Labs., Inc.,
        Yokomichi 41-1, Nagakute, Aichi, Japan, 4801192.
        {\tt\small \{ryotaro.shima,moriyasu,teruki.kato.tg\} @mosk.tytlabs.co.jp}}%
}

\begin{document}

\bstctlcite{IEEEexample:BSTcontrol}  

\maketitle
\thispagestyle{empty}
\pagestyle{empty}

\begin{abstract}

This paper presents a novel framework for the continuation method of model predictive control based on optimal control problem with a nonsmooth regularizer.
Via the proximal operator, the first-order optimality inclusion relation is reformulated into an equation system, to which the continuation method is applicable.
In addition, we present constraint qualifications that ensure the well-posedness of the proposed equation system.
A numerical example is also presented that demonstrates the effectiveness of our approach.

\end{abstract}

\section{INTRODUCTION}

Model predictive control (MPC) has been widely employed in the industry because it admits nonlinearity of the plant model and constraints to the states and inputs \cite{mpc, automotive_mpc}.
MPC implicitly determines control input using the optimal control problem (OCP).
Thus, nonlinear MPC requires online optimization of input series and often causes lack of optimality in the calculated inputs because of computation time limitation.
Although most real-time control methods, such as iterative LQR \cite{iLQR2004} and shooting methods \cite{shooting}, cause tracking delay to the time-varying reference points, the continuation method \cite{continuation, cgmres} accomplishes perfect tracking owing to prediction of time evolution of the optimal inputs with a linear equation system.
Such an algorithm is also referred to as a feedforward algorithm \cite{dorfler} or a prediction-correction algorithm \cite{tvo-survey}.
The continuation method requires an equation system that defines the optimal solution; see Subsection \ref{subsection:continuation}.

Nonsmooth optimization \cite{nonsmooth_nesterov} sharpens the landscape of the evaluation function at specific points and arranges the optimal solution at the edges.
A prominent example of nonsmooth optimization is sparse modeling \cite{sparse_modeling, naga}, which employs the $\ell_1$ norm as a penalty.
Nonsmooth optimization has also been explored in the context of OCP \cite{sparse_ocp, sparse_hjb, explicit}, with applications to spacecraft trajectory planning \cite{sparse_ThreeBody} and humanoid robot control \cite{sparse_ddp}.

A difficulty in the nonsmoothness of the evaluation function is that the first-order optimality condition becomes an inclusion relation.
To address this difficulty, the epigraphic reformulation technique is often used \cite{explicit, epi}, although it increases the computational burden due to the addition of slack variables.
A recent study \cite{sparse_ddp} discussed the approximation of the $\ell_1$ norm with differentiable functions, even though it lacks sparsity due to the approximation.
Another successful remedy to circumvent the difficulty is the proximal operator \cite{prox}, which is the solution of a certain tractable nonsmooth subproblem and has been utilized to construct a fixed point algorithm.
Previously, motivated by inverse optimization, the authors \cite{letter_docp} proposed differentiation of the sparse optimal control problem using the proximal operator.

This paper presents a novel framework for applying the continuation method to MPC with a nonsmooth regularizer.
Using the proximal operator, the first-order optimality inclusion relation of the nonsmooth OCP is reformulated into an equality, to which the continuation method is applied.
Furthermore, certain constraint qualifications are provided that ensure the well-posedness of the obtained equality.
Our proposed method enables tracking the optimal solution of nonsmooth OCP by solving linear equation system.
Finally, a numerical example is provided to illustrate the effectiveness of our proposed method.


The remainder of this paper is organized as follows.
Section II presents the preliminary contents.
Section III presents the continuation methods for the nonsmooth model predictive control.
Section IV presents a numerical example of our proposed method with sparse control for a nonlinear plant and illustrates the effectiveness of our proposed method.
Section V presents the conclusions of this paper.

\subsubsection*{Notation}

We denote a set of all positive real numbers by $\R_>$;
all non-negative real numbers by $\R_\geq$;
a vector (or matrix) $[a_{(1)}^\top \ \ldots \ a_{(k)}^\top]^\top$ by $[a_{(1)};\ldots;a_{(k)}]$;
the $i$-th element of a vector $v$ by $v_i$;
an element number of the finite-element set $\mathcal{S}$ by $|\mathcal{S}|$;
$\{1,\cdots,n\} \subset \mathbb{Z}$ by $\mathbb{Z}_n$;
a vector $[v^1;\cdots;v^T] \in \R^{nT}$, where $v^1, \ldots, v^T \in \R^n$, by $v^{1:T}$;
Hadamard product of $v$ and $w$ by $v\odot w$;
the $\ell_1$ norm of a vector $v$, namely, $\sum_i |v_i|$ by $\| v \|_1$;
a set $\{v+a \in \R^n \, | \, a \in A\}$, where $v\in \R^n$ and $A \subset \R^n$, by $v+A$;
and a $|\mathcal{I}|$-dimensional vector consisting of those elements of the $n$-dimensional vector $v$ whose indices are in $\mathcal{I}$ by $v_\mathcal{I}$, namely, for $\mathcal{I}= \{2,4,5\}$, $v_\mathcal{I}=[v_2; v_4; v_5] \in \R^3$.
Regarding the derivatives, we denote the gradient of a scalar function $f: \R^n \to \R$ at $x\in \R^n$ by $\nabla f(x) \in \R^n$, the Jacobi matrix of a vector-valued function $g: \R^n \to \R^m$ at $x\in \R^n$ by $\frac{\partial f}{\partial x}(x) \in \R^{m\times n}$, and the subderivative (see Subsection \ref{subsection:nonsmooth}) of a scalar function $f: \R^n \to \R$ at $x\in \R^n$ by $\partial f (x)$.

\section{PRELIMINARIES}

\subsection{Continuation Method \cite{continuationbook, cgmres}}
\label{subsection:continuation}

The continuation method tracks a solution of a nonlinear equation system dependent on a parameter.
Let the state of the plant $x(t) \in \R^n$, where $t$ denotes time.
Consider the following equation system regarding $z\in\R^k$ dependent on $x$:
\begin{align}
  F(z, x) = 0  .
  \label{equations}
\end{align}
Note that the solution $z^\mathrm{sol}$ implicitly depends on $x(t)$.
Instead of solving \eqref{equations}, the continuation method memorizes $z$ and achieves $z(t) \to z^\mathrm{sol}(x(t))$ as $t \to +\infty$ by forcing $z(t)$ to satisfy the following ordinary differential equation (ODE):
\begin{align}
  \frac{\mathrm{d}F(z(t), x(t))}{\mathrm{d}t} = - \zeta_\mathrm{c} F(z(t), x(t))  ,
\end{align}
where $\zeta_\mathrm{c}>0$ is a constant.
Taking the total derivative of the left hand side, we obtain the following ODE:
\begin{align}
  \frac{\mathrm{d} z}{\mathrm{d} t} &= d_\mathrm{c}  ,
  \label{ode:continuation} \\
  \frac{\partial F}{\partial z}(z, x) d_\mathrm{c}
  &=
  - \frac{\partial F}{\partial x}(z, x) \frac{\mathrm{d} x}{\mathrm{d} t} - \zeta_\mathrm{c} F(z, x)  .
  \label{linear_eqns_continuation}
\end{align}
Note that \eqref{linear_eqns_continuation} is a linear equation system with respect to $d_\mathrm{c}$, which is much easier to solve than the nonlinear equation system \eqref{equations}.

The following two aspects about the continuation method must be remarked.
One is that the continuation method requires an \textit{equation system}.
In OCP whose evaluation function is differentiable, the equation system \eqref{equations} can be derived as the first-order optimality condition (KKT condition) which, however, turns out to be an inclusion relation in nonsmooth OCP; see Subsection \ref{subsection:nonsmooth}.
The other is that it requires \textit{regularity} of the Jacobi matrix of $F$ with respect to $z$.

\subsection{Nonsmooth Calculus \cite{hilbert}}
\label{subsection:nonsmooth}

This subsection presents the calculus of the nonsmooth function to denote the first-order optimality condition for the nonsmooth optimization problem.
To specify the scope of nonsmoothness, a closed convex proper function is defined.
\begin{definition}[closed convex proper function]
  Consider a function $f:\R^n\to \R\cup\{+\infty\}$ and the following set\footnote{
    The set $\mathrm{epi}(f)$ is called an epigraph of $f$.
    }:
    \begin{align}
      \mathrm{epi} (f) \coloneqq \left\{ [x;t] \in \R^{n+1} \, \middle| \, f(x)\leq t \right\}  \notag
    \end{align}
    $f$ is said to be closed, convex, and proper (ccp) if $\mathrm{epi} (f)$ is nonempty, closed, and convex.
    $\fin$
\end{definition}
Note that all convex functions whose ranges are within $\R$ are ccp.
Moreover, the ccp property admits some continuous non-differentiable functions, such as $f(x)=\|x\|_1$, $f(x)=\|x\|_2$, and their finite summation.
This paper denotes non-differentiable ccp functions by \textit{nonsmooth functions}.

We introduce a subderivative of a nonsmooth function.
\begin{definition}[subderivative]
  Consider a ccp function $f:\R^n\to \R\cup\{+\infty\}$.
  For $x\in \R^n$ satisfying $f(x) < +\infty$, the following set
  \begin{align*}
    \partial f (x) \coloneqq \{ g \in \R^n \, | \, f(y) \geq f(x) + g^\top (y-x) \  \forall y \in \R^n \}.
  \end{align*}
  is said to be the subderivative of the function $f$ at $x$.
  $\fin$
\end{definition}

\begin{example}
    \label{example:subderivative}
    $\partial f(x)$ is a singleton, $\{ \nabla f (x) \}$, if $f$ is differentiable at $x$.
    This implies the subderivative is a generalization of the derivative.
    For $f(x) = \|x\|_1$ and $x=[0;-2;3]$, we obtain $\partial f(x) = \{[a;-1;1] \, | \, a \in [-1,1] \}$.
    $\fin$
\end{example}

The following lemma states the first-order optimality function for nonsmooth optimization problem.
\begin{lemma} [Theorem 16.3 in \cite{hilbert}]
    Consider the following nonsmooth optimization problem:
    \begin{align}
        \min_{x \in \R^n} f(x) + r(x),
    \end{align}
    where $f$ is a differentiable function and $r$ is nonsmooth function.
    Then, the optimal solution $x^\ast$ satisfies  the following first-order optimality condition:
    \begin{align}
        \nabla f(x^\ast) + \partial r(x^\ast) \ni 0 .
    \end{align}
\end{lemma}
Note that, the nonsmoothness of $r$ requires the subderivative, and the first-order optimality condition is obtained in the form of an inclusion relation.

\subsection{Proximal Operator \cite{prox, naga, hilbert}}

In this paper we utilize the proximal operator to eliminate the inclusion relation expressed by subderivatives.
The proximal operator is a tractable nonsmooth subproblem.
\begin{definition}[proximal operator]
  The proximal operator with respect to $r$ and $\gamma$ is defined as follows:
  \begin{align}
    \prox_\gamma^r (v) \coloneqq \argmin_{x\in \R^n} r(x) + \frac{1}{2\gamma} \|x-v\|^2 ,
    \label{prox}
  \end{align}
  where $r: \R^n \to \R\cup\{+\infty\}$ and a constant $\gamma >0$.
  $\fin$
\end{definition}

\begin{example}
    If $r(x) = \|x\|_1$, then $\prox_\gamma^r (v) = S^\gamma (v)$, where the $i$-th element of $S^\gamma(v)$ corresponds to $\mathrm{sgn}(v_i)\max\{|v_i|-\gamma, 0\}$.
    See also Table 2 in \cite{prox_list} for the list of explicitly obtained proximal operators.
    $\fin$
\end{example}

The following lemma is our main method for elimination of the inclusion relation caused by the nonsmoothness.

\begin{lemma}
  \label{lemma:prox}
  Given a function $r: \R^n \to \R\cup\{+\infty\}$, assume that there exists $\gamma_\mathrm{c}\in \R_>$ such that the evaluation function of the optimization \eqref{prox} is ccp.
  Then, the following equivalence
  \begin{align}
    w \in \gamma \partial r(x^\ast)
    \Longleftrightarrow
    x^\ast = \prox_\gamma^r (x^\ast + w) .
    \label{equiv}
  \end{align}
  holds for any $\gamma$ satisfying $0 < \gamma < \gamma_\mathrm{c}$.
\end{lemma}

\begin{proof}
  Let $v\in \R^n$ and $x^\ast \coloneqq \prox_\gamma^r 
 (v)$.
  The first-order optimality condition of the optimization in \eqref{prox} is as follows:
  \begin{align}
    \frac{v - x^\ast}{\gamma} \in \partial r (x^\ast)
    \label{prox_optimality}
  \end{align}
  The inequality $\gamma < \gamma_\mathrm{c}$ implies that the optimization problem in \eqref{prox} is strongly convex.
  Thus, the problem \eqref{prox} has the unique optimal solution, and $x^\ast$ is the solution if and only if the optimality condition \eqref{prox_optimality} holds at $x^\ast$.
  The right hand side of the relation \eqref{equiv} is equivalent to \eqref{prox_optimality}.
\end{proof}



\section{NONSMOOTH MODEL PREDICTIVE CONTROL}

\subsection{Nonsmooth Optimal Control Problem}

Consider the following discrete-time state space model:
\begin{align}
  x^{k+1} = f(x^k, u^k)    \ \ \ (k=1,2,\ldots)  ,
  \label{ssm}
\end{align}
where $k$ is time step, $x^k \in \R^n, u^k \in \R^m$ are the state and input at $k$, and $f : \R^n\times \R^m \to \R^n$ is twice differentiable.
Let the initial state be $x^1$.
With the $n_\mathrm{i}$ inequality constraints and $n_\mathrm{e}$ equality constraints, we consider the following nonsmooth optimal control problem:
\begin{align}
\begin{aligned}
  &\min_{u^{1:T}, x^{2:T+1}} ~ \phi(x^{T+1}) + \sum_{k=1}^T \left \{ L(x^k,u^k) + r(u^k) \right \} \\
  & \mathrm{subject \ to}\ \eqref{ssm}, \ g(u^k) \leq 0, \ h(u^k) = 0, \ k=1,\ldots,T,
\end{aligned}
\label{P1}
\end{align}
where $\phi:\R^n\to\R$, $L:\R^n\times\R^m\to\R$, $g:\R^m\to\R^{n_\mathrm{i}}$, $h:\R^m\to\R^{n_\mathrm{e}}$ are twice differentiable and $r:\R^n \to \R\cup\{+\infty\}$ is nonsmooth.
Throughout this paper, we suppose that the proximal operator of $r$ is obtained explicitly.

For convenience, define a function $H$ as follows:
\begin{align}
\begin{aligned}
  &H(x, u, \mu, \nu, p) \\
  &\coloneqq L(x,u) + \mu^\top g(u) + \nu^\top h(u) + p^\top f(x,u)  .
  \label{def_H}
\end{aligned}
\end{align}
Let the optimal control input be $\bar{u}^{1:T}$.
Then, the state $\bar{x}^{2:T+1}$ and the conjugate state $\bar{p}^{2:T+1}$ are algebraically determined by the following state space equation and conjugate equation:
\begin{gather}
  \bar{x}^{k+1} = f(\bar{x}^k, \bar{u}^k) \ \ (k=1,\ldots,T) ,
  \label{opt_ssm}
  \\
  \begin{multlined}
    \bar{p}^k = \nabla_x L (\bar{x}^k, \bar{u}^k) +
    \left (
      \frac{\partial f}{\partial x} (\bar{x}^k, \bar{u}^k)
    \right )^\top p^{k+1} \\
    (k=2,\ldots,T) ,  
  \end{multlined}
  \label{opt_adj}
  \\
  \nabla \phi (x^{T+1}) = p^{T+1} .
  \label{opt_terminal}
\end{gather}
We denote the dependency by $x^k(z,x^1), p^k(z,x^1)$.
Let the Lagrange multipliers of the inequality and equality constraints be $\mu^{1:T}, \nu^{1:T}$, respectively.
We denote the design variables by $z \coloneqq [u^{1:T}; \mu^{1:T}; \nu^{1:T}]$.
The following theorem is one of our main results.
\begin{theorem}
\label{thm_equation}
Consider the nonsmooth OCP \eqref{P1} and its optimal input $\bar{u}^{1:T}$.
Then, the state $x^{2:T+1}$ and the conjugate state $p^{2:T+1}$ are determined by $z$ and $x^1$ as in \eqref{opt_ssm}, \eqref{opt_adj} and \eqref{opt_terminal}.
Moreover, consider $\gamma \in \R_>$ satisfying the assumption in Lemma \ref{lemma:prox}.
Consider an arbitrary complementarity function\footnote{
  A function $\psi:\R\times\R \to \R$ is said to be a complementarity function if $\psi(a,b)=0\Leftrightarrow a\geq 0,b\geq 0,ab=0$ holds.
} $\psi$.
Then, there exists $\mu^{1:T}, \nu^{1:T}$ such that the following equation systems holds:
\begin{gather}
  c^{1:T}(z, x^1) = 0 ,
  \label{eq:F}\\
  c^k(z, x^1) \coloneqq \begin{bmatrix}
    u^k - \prox_\gamma^r
  \left (
    u^k - \gamma J^k (z,x^1)
  \right )\\
  \Psi(-g(u^k), \mu^k)\\
  h(u^k)
  \end{bmatrix}
  ,\label{eq:Fk}\\
  J^k (z,x^1) \coloneqq \nabla_u H(x^k(z,x^1), u^k, \mu^k, \nu^k, p^{k+1}(z,x^1))
  , \notag\\
  \Psi(g,\eta) \coloneqq [\psi(g_1, \eta_1) ; \ldots ; \psi(g_{n_\mathrm{i}}, \eta_{n_\mathrm{i}})] , \notag
\end{gather}
where $H$ is defined in \eqref{def_H}.
\end{theorem}

\begin{proof}
The first-order optimality condition for \eqref{P1} claims that there exists a Lagrange multiplier satisfying the following relation:
\begin{align}
    \partial r(u^k)
    + J(z,x^1)
    \ni 0 \ \
    &(k=1,\ldots,T) ,  
  \label{opt_u}
  \\
    g(u^k) \leq 0, \ \mu^k \geq 0, \ g(u^k) \odot \mu^k =0 \ \
    &(k=1,\ldots,T) ,
  \label{opt_inequality}
  \\
  h(u^k) = 0 \ \ 
  &(k=1,\ldots,T) .
  \label{opt_equality}
\end{align}
Owing to Lemma \ref{lemma:prox}, \eqref{opt_u} is reformulated into $ u^k - \prox_\gamma^r ( u^k - \gamma J(z, x^1) ) = 0 $ .
Moreover, \eqref{opt_inequality} is the complementarity condition and is reformulated into $\Psi(-g(u^k), \mu^k) = 0$ from the definition of a complementarity function.
\end{proof}

\subsection{Nonsmooth MPC by Continuation Method}

By selecting $c^{1:T}$ in \eqref{eq:F} as $F$ in \eqref{equations}, we establish a continuation method for nonsmooth MPC.
At the sampling time, the state of the plant is observed, and simultaneously the control input is injected into the plant.
We denote the state and input by $x^\mathrm{obs}$ and $u^\mathrm{in}$, respectively, and the predicted state at the next sampling time by $x^\mathrm{pred} \coloneqq f(x^\mathrm{obs},u^\mathrm{in})$.
Using the continuation method, the MPC controller updates the variable as $z \gets z+d_\mathrm{c}$, where $d_\mathrm{c}$ is the solution of the following linear equation system:
\begin{align}
  \begin{aligned}
    & \frac{\partial F}{\partial z}(z, x^\mathrm{obs}) d_\mathrm{c}
    \\
    & = - \frac{\partial F}{\partial x^1}(z, x^\mathrm{obs})(x^\mathrm{pred} - x^\mathrm{obs}) - \zeta_\mathrm{c} F(z, x^\mathrm{obs})  
  \end{aligned}
  \label{iter:continuation}
\end{align}

An algorithm of nonsmooth MPC using the continuation method is described in Algorithm \ref{alg1}.
What should be remarked is that the typical techniques to reduce the computational burden of solving linear equation system are also applicable to this algorithm.
For instance, the C/GMRES method \cite{cgmres} can efficiently solve \eqref{iter:continuation}.

\begin{figure}[t]
  \begin{algorithm}[H]
    \caption{Example of nonsmooth MPC using continuation method.}
    \label{alg1}
    \begin{algorithmic}[1]
      \REQUIRE initial value of $z \coloneqq [u^{1:T}; \mu^{1:T}; \nu^{1:T}]$ and iteration number $i_\mathrm{max}$ of Newton method
      \STATE use $u^1$ as a control input and memorize the observed state $x^\mathrm{obs}$.  \label{step:plant}
    \STATE Determine $x^{2:T+1}$ by \eqref{ssm} and $p^{2:T+1}$ by \eqref{opt_adj}, \eqref{opt_terminal}.
    \STATE Determine $d_\mathrm{c}$ by solving \eqref{iter:continuation}.
    \STATE $x^1 \gets f(x^\mathrm{obs}, u^1), \ z \gets z + d_\mathrm{c}$.
    \STATE Go Step \ref{step:plant} at sampling time.
    \end{algorithmic}
  \end{algorithm}
  \vspace{-15pt}
\end{figure}

\subsection{Well-posedness of Linear Equation System}

This subsection discusses the well-posedness of \eqref{iter:continuation}.

\subsubsection{Differentiability of $F$}

The linear equation system \eqref{iter:continuation} requires the differentiability of $F$, namely, $\prox_\gamma^r$ and $\Psi$ in \eqref{eq:Fk}.
Regarding $\prox_\gamma^r$, the following lemma holds.
\begin{lemma}
Under the assumption in Lemma \ref{lemma:prox}, $\prox_\gamma^r(v)$ is differentiable at almost all $v\in \R^n$.
\end{lemma}
\begin{proof}
$\prox_\gamma^r(v)$ is cocoercive (see Example 22.7 in \cite{hilbert}) and therefore Lipschitz (see Example 22.7 in \cite{hilbert}).
Thus, Rademacher's theorem implies the statement.
\end{proof}

Regarding $\Psi$, when $\psi$ is Fischer--Burmeister function $\psi_\mathrm{FB}(a,b) \coloneqq a+b-\sqrt{a^2+b^2}$, it is differentiable almost everywhere.
Therefore, the differentiability does not cause a practical problem.
One theoretical remedy is to replace the derivative of $F$ in \eqref{iter:continuation} with an element of the Clarke subdifferential, under the assumption that $F$ is locally Lipschitz.

\subsubsection{Regularity of Jacobi matrix}

The linear equation system \eqref{iter:continuation} is feasible as long as the Jacobi matrix is regular.
The regularity requires the uniqueness of the solution of the nonlinear equation system \eqref{eq:F}.
The uniqueness of $u^{1:T}$ is guaranteed if the OCP \eqref{P1} is strongly convex.
In the following, we discuss the uniqueness of Lagrange multipliers.

Consider $u \in \R^m$ satisfying the inequality and equality constraints.
We denote an index set of active constraints by $\mathcal{I}(u) \coloneqq \{ i \in \Z_{n_\mathrm{i}} \mid g_i(u) = 0 \}$.
Define a cone $N(u)$ and a matrix-valued function $J(u)$ as follows:
\begin{align}
  N (u) &\coloneqq
  \left \{
    J (u) ^\top [\mu ; \nu]
    \in \R^m
    \, \middle | \,
    \begin{aligned}
        \mu &\in \R^{n_\mathrm{i}}, \ 
        \nu \in \R^{n_\mathrm{e}}, \\
        \mu_i &\geq 0 \ \ \mathrm{if} \ i \in \mathcal{I}(u), \\
        \mu_i &= 0 \ \ \mathrm{if} \ i \notin \mathcal{I}(u)    
    \end{aligned}
  \right \}
  ,
  \\
  J (u) &\coloneqq
  \left [
    \frac{\partial g}{\partial u} (u) ; \frac{\partial h}{\partial u} (u)
  \right ]  .
\end{align}

Let us assume linear independence constraint qualification.
\begin{assumption}[LICQ]
  \label{cq:licq}
  The vectors $\nabla g_i(u), \ i \in \mathcal{I}(u)$, $\nabla h_i(u), \ i \in \Z_{n_\mathrm{e}}$ are linearly independent.
  $\fin$
\end{assumption}

In addition, consider the following constraint qualification that implies independence of $\partial r(u)$ and $N(u)$.
\begin{assumption}
  \label{cq:cap}
  Each element $z \in \partial r(u) + N(u)$ is uniquely decomposed into a sum $z = v + w, \, v \in \partial r(u), \, w \in N(u)$.
  $\fin$
\end{assumption}

\begin{example}
    Let $m=1$ and $r(u) = \|u\|_1$, and consider a box constraint $-1 \leq u \leq 1$, which is realized by $g(u) = [u-1; -u-1]$.
    In this case, Assumption \ref{cq:cap} holds at any point $u$.
    For instance, if $\bar{u}=0$, then $\partial r(\bar{u}) = [-1,1]$ (see also Example \ref{example:subderivative}), $N(\bar{u}) = \{0\}$, and $\partial r(\bar{u}) + N(\bar{u}) = [-1,1]$, and consequently $z \in \partial r(\bar{u}) + N(\bar{u})$ implies $z \in \partial r(\bar{u})$ and $0 \in N(\bar{u})$.
    In addition, if $\bar{u}=1$ then $\partial r(\bar{u}) = \{1\}$, $N(\bar{u}) = \{a\,| \, a \geq 0\}$, and $\partial r(\bar{u}) + N(\bar{u}) = \{a\,| \, a \geq 1\}$, and consequently $z \in \partial r(\bar{u}) + N(\bar{u})$ implies $1 \in \partial r(\bar{u})$ and $z-1 \in N(\bar{u})$.
    $\fin$
\end{example}

\begin{theorem}
  \label{thm_optimality}
  Suppose Assumptions \ref{cq:licq} and \ref{cq:cap} hold for each $\bar{u}^1, \ldots, \bar{u}^T$, where $\bar{u}^{1:T}$ is the optimal control input of \eqref{P1}.
  Then, under the assumption in Theorem \ref{thm_equation}, there uniquely exist Lagrange multipliers $\mu^{1:T}, \nu^{1:T}$ satisfying \eqref{eq:F}.
\end{theorem}

\begin{proof}
  The first-order optimality condition \eqref{opt_u}--\eqref{opt_equality} is also expressed as follows:
  \begin{gather}
    \partial r(\bar{u}^k) + N(\bar{u}^k)
    \ni a^k \ \ (k=1,\ldots,T), \\
    a^k \coloneqq - \nabla_u L (\bar{x}^k, \bar{u}^k)
    - \left (
      \frac{\partial f}{\partial u} (\bar{x}^k, \bar{u}^k)
    \right )^\top \bar{p}^{k+1}
  \end{gather}
  Assumption \ref{cq:cap} implies that, for each $z^k$, there uniquely exist $v^k \in \partial r (\bar{u}^k)$ and $w^k \in N(\bar{u}^k)$ such that $v^k + w^k = z^k$.
  Assumption \ref{cq:licq} implies, for each $w^k \in N(\bar{u}^k)$, there uniquely exist $\mu^k \in \R^{n_\mathrm{i}}, \nu^k \in \R^{n_\mathrm{e}}$ such that \eqref{opt_inequality}, \eqref{opt_equality}, and $w^k = J (\bar{u}^k) ^\top [\mu ; \nu]$ hold.
  Therefore, the Lagrange multipliers $\mu^k, \nu^k$ satisfying the first-order optimality condition are unique.
\end{proof}


\section{NUMERICAL EXAMPLE: SPARSE MPC}

%
%

\subsection{Problem Setting}

Let the sampling period be $\varDelta t=0.05$, and consider the state space model \eqref{ssm} determined by the following function $f$:
\begin{gather*}
  f(x,u) = x + \varDelta t \tilde{f}(x) + \varDelta t Bu, \\
  \tilde{f}(x) \coloneqq \begin{bmatrix}
    x_3\\
    x_4\\
    -0.1x_1 - 0.1\cosh(0.1x_2)x_3\\
    -0.2x_2 - 0.2\cosh(0.1x_1)x_4 + 0.1x_4\\
    -0.3x_5 + \tanh(x_3) + \tanh(x_4)
  \end{bmatrix} ,  \\
  B \coloneqq \begin{bmatrix}
    0 & 0 & 1 & 0 & 0\\
    0 & 0 & 0 & 1 & 0
  \end{bmatrix}^\top .
\end{gather*}
Let the initial state be $x=[6;-8;3;-2;5]$.
Consider no equality constraints of the input but the upper and lower bounds of inputs $-1 \leq u_i \leq 1 \ (i=1,2)$.
Let $T=60$ and consider the following cost functions $L, \phi$:
\begin{align}
  L(x,u) = \frac{1}{2} x^\top x + u^\top u , \ 
  \phi(x) = \frac{1}{10} x^\top x .
\end{align}
To induce sparsity to the control input, we set the nonsmooth regularizer $r$ to be $4\|\cdot\|_1$ and select $\gamma=0.5$.
Note that $\prox_\gamma^r (v) = S^\beta (v)$, where $\beta=\gamma/4$.

\subsection{Conventional Method for Comparison}

A paper \cite{sparse_ddp} proposed an equalization of optimality condition by approximating the nonsmooth function $r$ by differentiable function $\tilde{r}$.
More specifically, instead of \eqref{opt_u}, it considers the following equation:
\begin{align}
  \nabla \tilde{r}(u^k)
  + J^k(z, x^1)
  = 0 \ \ (k=1,\ldots,T)  .
\end{align}
Thus, instead of \eqref{eq:F}, the following equation system is utilized for continuation method:
\begin{align}
  \tilde{c}^{1:T}(z,x_1)=0 ,  \label{conv_equation} \\
  \tilde{c}^k(z, x^1) \coloneqq \begin{bmatrix}
    \nabla \tilde{r}(u^k) + J^k (z,x^1)\\
    \Psi(-g(u^k), \mu^k)\\
    h(u^k)
  \end{bmatrix}
\end{align}
In this numerical example, $r(u)=4\|u\|_1$ is replaced with a differentiable function $\tilde{r}(u) = 4\varepsilon \sum_{i=1}^2 \log\cosh(u_i/\varepsilon), \varepsilon=10^{-2}$.

\subsection{Results of Control and Discussion}

In both our proposed method and the conventional method, we update $z$ using the solution of the linear equation system \eqref{iter:continuation} with $\zeta_\mathrm{c}=0.4$, and additionally update it using the Newton method with step size $0.8$.
We also let the complementarity function be the Fischer--Burmeister function $\psi_\mathrm{FB}(a,b) \coloneqq a+b-\sqrt{a^2+b^2}$.
Initialization of $z$ is conducted using the Newton method.

The control inputs calculated by our proposed method and the conventional method are depicted in Fig.~\ref{fig:inputs}, whereas those by the conventional method are depicted in Fig.~\ref{fig:inputs_conv}.
Our proposed method attained sparse control, whereas the inputs by the conventional method did not produce zero inputs but produced approximately zero inputs.
Furthermore, our proposed method switched off $u_1$ at around $t=4.5$ and $u_2$ at around $t=3$, whereas the inputs calculated by the conventional method fluctuates at that time.

\begin{figure}[t]
  \centering
  \includegraphics[width=0.8\linewidth]{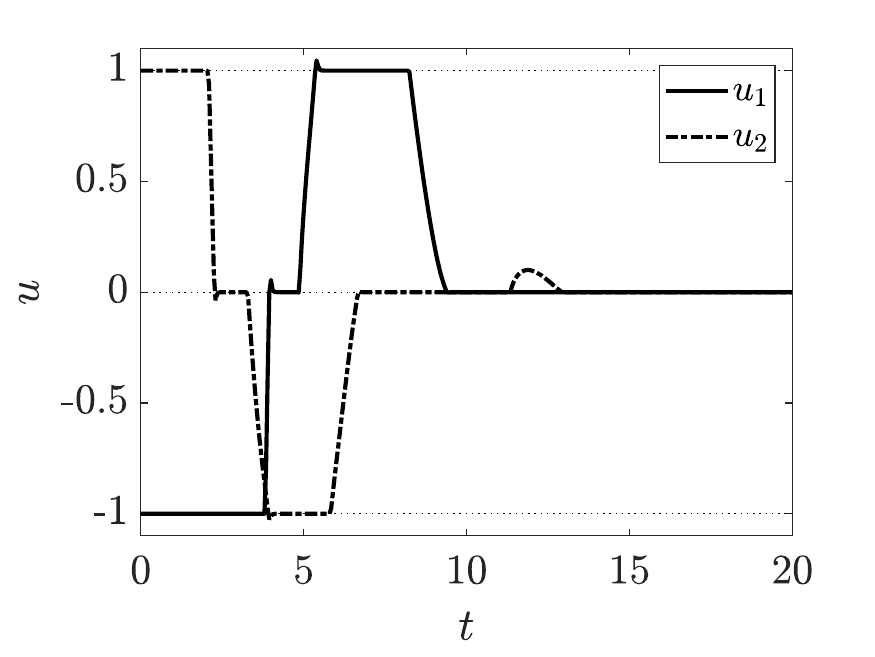}
  \caption{Inputs calculated by the proposed method.  \label{fig:inputs}}
\end{figure}

\begin{figure}[t]
  \centering
  \includegraphics[width=0.8\linewidth]{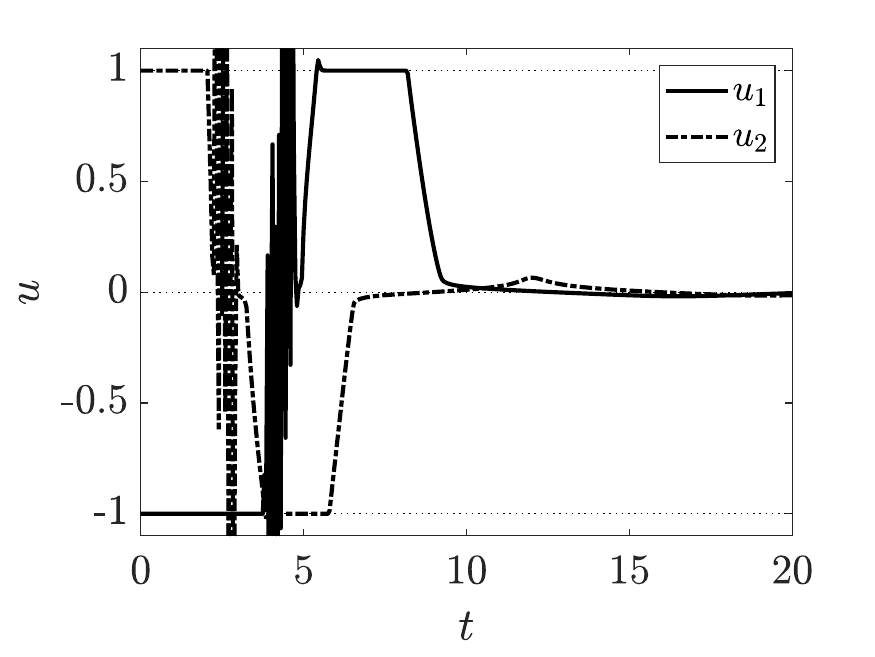}
  \caption{Inputs calculated by the conventional method.  \label{fig:inputs_conv}}
\end{figure}

Our method is comparably effective for nonsmooth MPC because it has no tradeoff inherent to the conventional method.
Due to approximation, the conventional method loses the sparsity structure, which is intended when the regularizer is introduced.
Furthermore, regarding the stability of the input calculation, the conventional method deteriorates the condition number of the linear equation system \eqref{linear_eqns_continuation} around the nonsmooth points because the slope of the cost function changes drastically there.
Therefore, the accuracy of the approximation is selected under the tradeoff between the loss of sparsity and the instability of calculation.
Meanwhile, our proposed method accomplishes sparse control because it does not approximate the $\ell_1$ norm.
In addition, in our proposed method, the proximal operator switches off the unnecessary directions of the solution of \eqref{linear_eqns_continuation} and hence causes no such deterioration of condition number.
Therefore, our proposed method does not inherit the tradeoff of the conventional method.

\begin{figure}[t]
  \centering
  \includegraphics[width=0.8\linewidth]{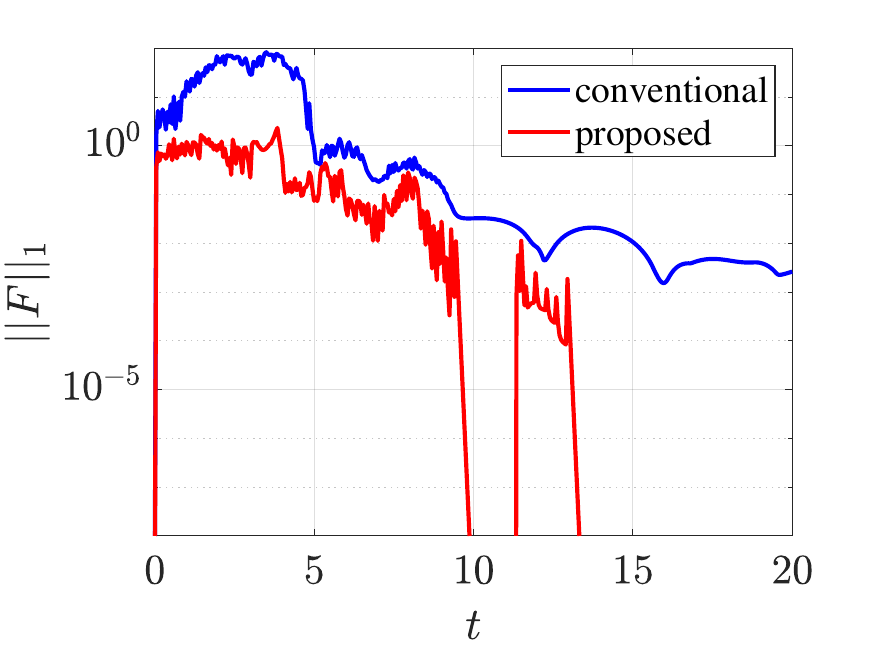}
  \caption{Residual of the equation system.  \label{fig:fval}}
\end{figure}

\begin{figure}[t]
  \centering
  \includegraphics[width=0.8\linewidth]{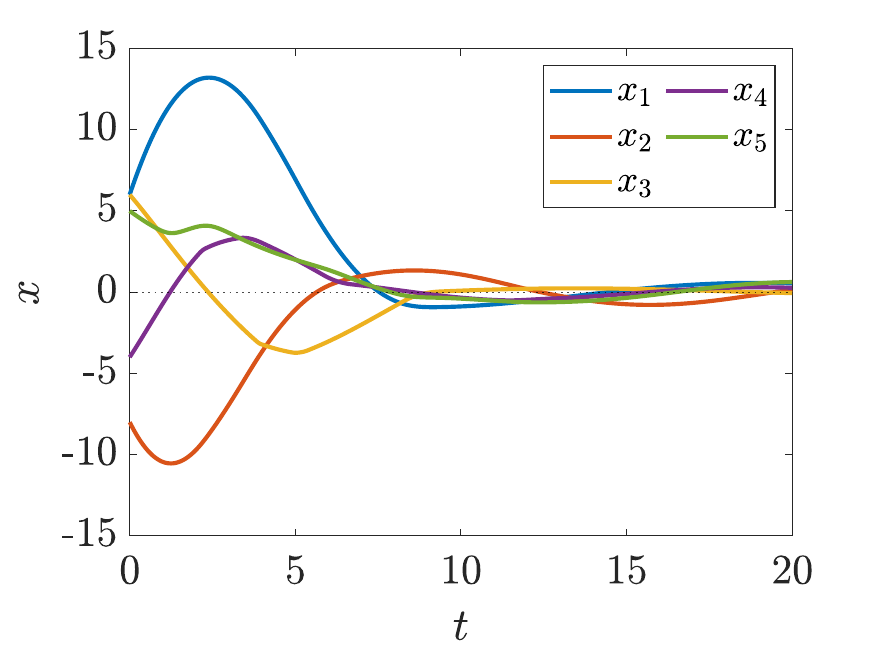}
  \caption{Trajectory of states controlled by the proposed method.  \label{fig:states}}
\end{figure}

What should be remarked in addition is that our proposed method spends as much time as the conventional method to solve the linear equation system \eqref{linear_eqns_continuation}.
The dimension of the linear equation system \eqref{linear_eqns_continuation} is the same in both our proposed method and the conventional method since, unlike \cite{explicit}, our proposed method does not introduce any slack variables to address the nonsmoothness.

A logarithmic plot of the $\ell_1$ norm of $c^{1:T}(z)$ and $\tilde{c}^{1:T}(z)$ (which is also referred to as $F(z)$ in \eqref{equations}) is depicted in Fig.~\ref{fig:fval}, which illustrates that our method has less residual than the conventional method.
Note that the reason why both control inputs violate the box constraint is because the continuation method calculates inputs in a suboptimal manner; note that it achieves $z(t) \to z^\mathrm{sol}(x(t))$ instead of solving \eqref{equations} (see also Subsection \ref{subsection:continuation}).
In practice, the calculated input can simply be clipped to the box to satisfy the constraint.

The trajectory of states controlled by our proposed method is depicted in Fig.~\ref{fig:states}.
The states almost converged to the origin at $t=20$.
Moreover, because the plant system is asymptotically stable under zero input, the state will converge to the origin after $t=20$.

\section{CONCLUSION}

This paper proposed a novel framework for the continuation method of nonsmooth MPC.
Via the proximal operator, we proposed an equation system which is equivalent to the first-order inclusion relation induced by the nonsmoothness of the OCP.
Moreover, constraint qualifications were also presented that ensure the uniqueness of the Lagrange multipliers satisfying the proposed equation system.
Finally, we presented a numerical example to show that our proposed method outperforms the conventional method.

\bibliographystyle{IEEEtran}
\bibliography{references}





\end{document}